\documentclass[letterpaper,11pt,reqno]{amsart} 
\usepackage[portrait,margin=1in]{geometry} 

\usepackage{mathrsfs,xfrac} 
\usepackage[colorlinks=true,linkcolor=blue,citecolor=blue,urlcolor=blue]{hyperref} 
\usepackage{amsmath,amssymb,amsthm,amsfonts,amsbsy,latexsym,dsfont,color,graphicx,enumitem}

\usepackage{caption}
\usepackage{regexpatch}
\usepackage{comment}
\usepackage{todonotes}
\makeatletter
\xpatchcmd{\@todo}{\setkeys{todonotes}{#1}}{\setkeys{todonotes}{inline,#1}}{}{}
\makeatother
\usepackage{algorithm}
\usepackage{algpseudocode}

\newtheorem{thm}{Theorem}[section]
\newtheorem{lem}[thm]{Lemma}
\newtheorem{cor}[thm]{Corollary}
\newtheorem{prop}[thm]{Proposition}

\newtheorem{obs}[thm]{Observation}

\newtheorem{ex}[thm]{Example}

\newtheorem{ques}[thm]{Question}

\renewcommand{\le}{\leqslant}  
\renewcommand{\ge}{\geqslant}

\newcommand{\abs}[1]{\left\vert#1\right\vert}

     \let\gl=\lambda

\newcommand{\cF}{\mathcal{F}}

\newcommand{\mvU}{\boldsymbol{U}}

\newcommand{\bN}{\mathbb{N}}
\newcommand{\bR}{\mathbb{R}}

\newcommand{\bZ}{\mathbb{Z}}        
\DeclareMathOperator{\pr}{\mathds{P}}

\DeclareMathOperator{\Aut}{Aut}

\newcommand{\w}{\mathbf{w}}
\newcommand{\fmsf}{\mathrm{FMSF}}
\newcommand{\wmsf}{\mathrm{WMSF}}
\newcommand{\fmax}{\mathrm{FMaxSF}}
\newcommand{\wmax}{\mathrm{WMaxSF}}

\usetikzlibrary{calc}

\begin{document}
\title{$\beta$-Skewed Maximal Spanning Forests}
\author[Chatterjee]{Shirshendu Chatterjee}
\address{Shirshendu Chatterjee\\
Department of Mathematics\\
The City College of New York and CUNY Graduate Center\\
New York, NY,
USA
}
\email{shirshendu@ccny.cuny.edu}
\author[Chen]{Yang Chen}
\address{Yang Chen\\
Department of Mathematics\\
The City College of New York\\
New York, NY,
USA
}
\email{yang.chen57@stu-mail.ccny.cuny.edu}
\author[Terlov]{Grigory Terlov}
\address{
Grigory Terlov \\
Department of Statistics and Operations Research\\
University of North Carolina\\
Chapel Hill, NC \\
USA
}
\email{gterlov@unc.edu}
\begin{abstract}
The Free $\w$-Maximal Spanning Forest (FMaxSF) is a weighted generalization of the classical Free Minimal Spanning Forest ($\fmsf$) that is able to detect nonhyperfiniteness in percolation on nonunimodular graphs. We introduce a parameterized family of invariant random spanning forests that interpolates between these models. For every finite positive value of the parameter $\beta$, the construction retains many of the desired properties of $\fmax$ while also admitting the finite-subtree forcing property of $\fmsf$. We study local limits of these forests and, in particular, show that the small-$\beta$ limit of the wired variant coincides with $\fmsf$ if and only if $p_h=p_u$, where $p_h$ is the threshold for the existence of heavy clusters and $p_u$ is the uniqueness threshold for Bernoulli$(p)$ percolation. Finally, we show that the Free and the Wired $\w$-Maximal Spanning Forests may coincide even if $p_h<p_u$, providing a negative answer to a question from \cite{TT25}.
\end{abstract}

\maketitle

%%%%%%%%%%%%%%%%%%%%%%%%%%%%%%%%%%%%%%%%%%%%%%%%%%%%%%%%%%
\section{Introduction}
%%%%%%%%%%%%%%%%%%%%%%%%%%%%%%%%%%%%%%%%%%%%%%%%%%%%%%%%%%
Consider an infinite connected locally finite quasi-transitive graph $G=(V,E)$. Fixing a closed subgroup $\Gamma$ of automorphisms $\Aut(G)$ that acts quasi-transitively on $G$ equips $V$ with a relative weight function $\w_\Gamma$, which in a sense quantifies the asymmetry of the action of $\Gamma$ on $G$. In particular, this function is constant on each $\Gamma$-orbit if and only if $\Gamma$ is unimodular.
Since Cayley graphs are always unimodular, and due to the availability of various tools, this case has been a primary focus in percolation literature. Indeed, when the weights are nontrivial, many classical techniques do not translate, requiring one to either exploit the structure induced by the weights or develop weighted analogs of classical unimodular tools and objects.
The systematic study of percolation in the nonunimodular setting began in \cite{HaggstromTrees,Haggstrom99,BLPS99inv,Timar06nonu} and recently saw a lot of progress \cite{Pengfei18,Hutchcroft20, Tang_WUFS_light,CTT22,HutchcroftPan24rel,TT25,HeavyRepulsion,Bowen_Oneended}.

The Free Minimal Spanning Forest ($\fmsf$) is a well-studied object in discrete probability; we refer the reader to \cite{LPS06msf,LyonsBook} for an overview and the significance of this model.
The main object introduced in \cite{CTT22} is a weighted analog of this model, called the Free $\w$-Maximal Spanning Forest. The idea behind their construction was to generalize $\fmsf$ to produce a suitable infinitely ended subforest of a weighted graph in contexts of measured group theory and percolation on nonunimodular graphs. 

\subsection{Minimal and maximal spanning forests}
The classical minimal spanning forests, the $\w$-maximal spanning forests of \cite{CTT22}, and the $\beta$-skewed maximal spanning forests introduced here arise from the same cycle-cutting algorithm, applied with different linear orders on the edges of the underlying graph. First, given a vertex weight function $\w:V\to\bR^+$, we extend it to the edges by setting 
\[
\w((x,y)):=\min\{\w(x),\w(y)\}.
\]
Next given random i.i.d.\ labels $\mvU=(U_e)_{e\in E}$ each with $\mathrm{Uniform}[0,1]$ distribution and $\beta\in[0,\infty)$, for each $e\in E$ we define 
\[
X_{e,\beta}:=\w(e)^\beta(1-U_e).
\]

Consider the following orders on $E$, which are a.s.\ linear for all $\beta\ge0$:
\begin{align*}
    e&\prec_0 e'  \,\,\qquad:\Longleftrightarrow\quad1-U_e<1-U_{e'},\\
    e &\prec_{\w} e'\qquad:\Longleftrightarrow\qquad\w(e)<\w(e') \quad \text{or}\quad \bigl(\w(e)=\w(e') \;\;\mbox{and}\;\; 1-U_e<1-U_{e'}\bigr),\\
    e &\prec_\beta e'\,\qquad:\Longleftrightarrow\qquad X_{e,\beta}<X_{e',\beta}   \quad \text{or}\quad \bigl(X_{e,\beta}=X_{e',\beta}  \;\;\mbox{and}\;\; 1-U_e<1-U_{e'}\bigr).
\end{align*}
Here, for each fixed value of $\beta$ the tie-breaker rule in $\prec_\beta$ is a.s.\ redundant; however since we will be interested in varying the parameter it ensures that the order is linear for all $\beta\ge0$ simultaneously.

For each of these orders, the \textbf{free} version of the forest is obtained by simultaneously deleting the least edge in every finite cycle, while the \textbf{wired} version additionally deletes every edge that is least on some bi-infinite simple path. The resulting forests produced by 
\begin{itemize}
    \item[$\prec_0$] are called the \textbf{Free}, resp.\ \textbf{Wired}, \textbf{Minimal Spanning Forest} and are denoted by\\ $\fmsf(G,\mvU)$, resp.\ $\wmsf(G,\mvU)$;
    \item[$\prec_{\w}$] are called the \textbf{Free}, resp.\ \textbf{Wired}, $\w$-\textbf{Maximal Spanning Forest} and are denoted by\\$\fmax(G,\mvU)$, resp.\ $\wmax(G,\mvU)$;
    \item[$\prec_{\beta}$] are called the \textbf{$\beta$-Skewed} \textbf{Free}, resp.\ \textbf{Wired}, $\w$-\textbf{Maximal Spanning Forest} and are denoted by $\mathfrak{F}_\beta(G,\mvU)$, resp.\ $\mathfrak{W}_\beta(G,\mvU)$.
\end{itemize} 
When the dependency on the graph $G$ or labels $\mvU$ is not important to highlight, we omit them from notation. Finally, we note that one usually defines $\fmsf$ and $\wmsf$ by deleting the edge with the largest $U_e$ label in each cycle or bi-infinite path. However, for consistency in our presentation, it is easier to work with $\prec_0$. In particular, $\mathfrak{F}_0(G,\mvU)=\fmsf(G,\mvU)$ and $\mathfrak{W}_0(G,\mvU)=\wmsf(G,\mvU)$.

There are two main applications of minimal spanning forests that motivate these weighted constructions. 
First, the behaviors of minimal spanning forests on a graph $G$ characterize properties of Bernoulli$(p)$ percolation on $G$. For example, \cite[Proposition 3.6]{LPS06msf} shows that $\fmsf(G)$ and $\wmsf(G)$ coincide a.s.\ if and only if for a.e.\ $p\in(0,1)$ there is at most one infinite cluster in Bernoulli$(p)$ percolation on $G$. 
Second, $\fmsf$ can be applied to configurations of invariant percolation to extract infinitely ended subforests, e.g.\ \cite[Lemma 7.4]{BLPS99inv}. This enabled results in measured group theory \cite{GL09,Thom_Expected_Degree_FMSF,Tarski,GheysensMonod} and percolation \cite[Theorem 1.3]{BLPS99inv} (see also \cite[Theorem 8.21]{LyonsBook}), with similar recent uses in measurable combinatorics \cite{BTT_Whitney}. For various applications of $\wmsf$ in measurable combinatorics, see e.g.~\cite{CMTD_Brooks,berczi2026cycle}.

In the nonunimodular setting or in the measure class preserving setting in measured group theory (which is in many ways an analogous assumption) the existence of infinitely many ends in a forest is not as strong of a property. Indeed, although graphs such as the grandparent graph (see Example~\ref{ex:gp}) have infinitely many ends, they are still hyperfinite. To address this phenomenon, the authors of \cite{CTT22,AnushRobin} developed appropriate notions of weighted ends (called $\w$-nonvanishing ends, see Section~\ref{sec:prelim}) that enable generalizations of several foundational results from the unimodular and probability measure preserving settings. 

Motivated by the use of $\fmsf$ in the proof of the Gaboriau--Lyons theorem \cite{GL09}, in \cite{CTT22} Chen, Terlov, and Tserunyan introduced $\fmax$ to produce a subforest with infinitely many $\w$-nonvanishing ends under appropriate assumptions. To be precise, in the measured-group theoretic setting they assumed that the a.e.~ component of the underlying graph has at least $3$ $\w$-nonvanishing ends. At the same time, in the context of invariant percolation they imposed an additional assumption of existence of a $\w$-trifurcation vertex (see Section~\ref{sec:prelim}).
Interestingly, our Theorem~\ref{thm:prod} can be used to give an example of a measurable nonhyperfinite graph on which $\fmax$ has at most $2$ $\w$-nonvanishing ends, showing the necessity of the multi-endedness assumption for the methods of \cite{CTT22}.
In \cite{TT25}, Terlov and Tim\'ar showed further applications of $\fmax$ in percolation theory, in particular, to establish the continuity of the percolation phase transition at $p_h$ for weighted-nonamenable graphs.

The main property of the order $\prec_{\w}$ is that it gives the weights absolute priority. This choice is natural in the measured-group-theoretic setting, where the random labels $\mvU$ are replaced by an arbitrary Borel linear order. On the one hand, it helps because it makes it easier for the forest to capture the weighted geometry; on the other hand, the resulting forest loses the local flexibility inherent to $\fmsf$. To be precise, $\fmax$ does not have a convenient property of $\fmsf$ that any finite subtree of $G$ can be made a part of $\fmsf$ with positive probability.
Moreover, weakening the role of the random labels also weakens the connection to Bernoulli$(p)$ percolation. For example, \cite[Question 7.4]{TT25} asks whether maximal spanning forests characterize the nonuniqueness of heavy clusters in Bernoulli percolation (analogous to \cite[Proposition 3.6]{LPS06msf}). To our surprise, the answer turns out to be negative; see Theorem~\ref{thm:prod}.

The skewed order introduced in this paper relaxes this absolute priority and introduces a competition between weights and random labels. We show that, for every
$\beta\in(0,\infty)$, the resulting free forest admits the desired finite-subtree forcing property while retaining many key properties of $\fmax$. In particular, Theorem~\ref{thm:main_CTT} allows us to remove the assumption on a $\w$-trifurcation vertex from \cite[Theorem 1.10]{CTT22}. 

Since the parameter $\beta$ also places skewed maximal forests in a common family, they are in a way analogous to finite-temperature variants for statistical mechanics models. Note that on any finite set of edges the linear order $\prec_\beta$ interpolates between $\prec_0$ and $\prec_\w$ (see Lemma~\ref{order_alignment}). 
However, as our Theorems~\ref{thm:local_lim_W}--\ref{thm:local_lim_F_large_discr} show, the local limits of the $\beta$-skewed maximal forests could exhibit subtle behavior.

Finally, in light of our negative answer to the question \cite[Question 7.4]{TT25}, the most surprising consequence of our results on the local limits of $\mathfrak{W}_{\beta}$ is that $\mathfrak{W}_{0+}$ recovers a characterization of the heavy-cluster nonuniqueness phase, based on whether it coincides with $\fmsf$; see Theorem~\ref{thm:ph_vs_pu}.

\begin{figure}[htb]
\begin{center}
	\begin{tikzpicture}[thick, scale=0.7]
	\node at (-9,6) {$\beta=0$};
    \node at (-4,6) {$\fmsf$};
    \node at (-2,6) {$\supseteq$};
    \node at (0,6) {$\mathfrak{W}_{0+}$};
    \node at (2,6) {$\supsetneq$};
    \node at (4,6) {$\wmsf$};

    \node at (-9,3) {$\beta>0$};
    \node at (-3,3) {$\mathfrak{F}_{\beta}$};
    \node at (0,3) {$\supseteq$};
    \node at (3,3) {$\mathfrak{W}_{\beta}$};

    \node at (-9,0) {$``\beta=\infty"$};
    \node at (-3,0) {$\fmax$};
    \node at (0,0) {$\supseteq$};
    \node at (3,0) {$\wmax$};
    \draw [->] 
        (2.7,3.5) 
            to node [right] {\,Theorem~\ref{thm:local_lim_W}} 
        (0.2,5.6);
    \draw [->] 
        (3,2.5) 
            to node [right] {Theorem~\ref{thm:local_lim_W}} 
        (3,0.5);
    \draw [->,dashed] 
        (-3,2.5) 
            to node [left] {\shortstack[c]{Question~\ref{ques:local_lim_F_large}\\Theorems~\ref{thm:local_lim_F},~\ref{thm:local_lim_F_large_discr}}} 
        (-3,0.5);
    \draw [->,dashed] (-2.7,3.5) to (-0.2,5.6);
    \draw [->,dashed] (-3,3.5) to (-4,5.6);
    \node at (-2.4,5) {Obs.~\ref{prop:prod_small_beta}};
    \draw (-4,6.5) -- (-4,6.8);
    \draw (-0.1,6.5) -- (-0.1,6.8);
    \draw [-] 
        (-4,6.8) 
        to node [above] {\shortstack[c]{$=\,\Leftrightarrow p_h=p_u$\\ Theorem~\ref{thm:ph_vs_pu}}} 
        (-0.1,6.8);
    \draw (-3,-0.5) -- (-3,-0.8);
    \draw (3,-0.5) -- (3,-0.8);
    \draw [-] 
        (-3,-0.8) 
        to node [below] {\shortstack[c]{$=$ if $\w$-amenable\\ Proposition~\ref{prop:amen_F=W}}} 
        (3,-0.8);

    \begin{scope}[reset cm]
      \path let \p1 = (current bounding box.east),
                \p2 = (current bounding box.west),
                \n1 = {max(\x1,-\x2)}
            in (-\n1,0) (\n1,0);
    \end{scope}
    \end{tikzpicture}	
\end{center}
\caption{This diagram depicts local convergences of the $\beta$-skewed maximal spanning forests on transitive nonunimodular graphs. Solid arrows denote established limits in full generality, while dashed ones denote limits where we needed additional assumptions. Observation~\ref{prop:prod_small_beta} provides examples where the local limit of $\mathfrak{F}_{\beta}$, as $\beta\searrow0$, attains either $\fmsf$ or $\mathfrak{W}_{0+}$.}
\end{figure}
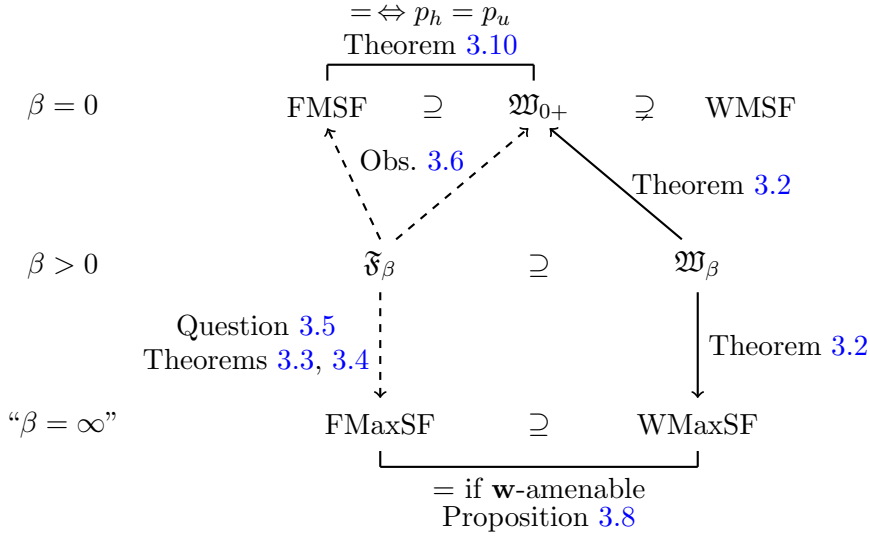
%%%%%%%%%%%%%%%%%%%%%%%%%%%%%
\subsection{Organization}
%%%%%%%%%%%%%%%%%%%%%%%%%%%%%
In Section~\ref{sec:prelim} we briefly review all required definitions and key examples. In Section~\ref{sec:main_res} we state our main results and provide additional context and discussion. The rest of the paper is dedicated to the proofs: in Section~\ref{sec:flexibility_F} we establish key properties of $\mathfrak{F}_{\beta}$ and prove our variant of \cite[Theorem 1.10]{CTT22}; Sections~\ref{sec:loc_lim_W} and~\ref{sec:loc_lim_F} are dedicated to local limit results for $\mathfrak{W}_{\beta}$ and $\mathfrak{F}_{\beta}$, respectively; in Section~\ref{sec:char_ph_pu} we present the proofs of the results related to the heavy-cluster nonuniqueness phase.
%%%%%%%%%%%%%%%%%%%%%%%%%%%%%
\subsection{Acknowledgments}
%%%%%%%%%%%%%%%%%%%%%%%%%%%%%
We thank \'Ad\'am Tim\'ar and G\'{a}bor Pete for insightful conversations.

S.C.\ was supported in part by the NSF grant DMS-2154564.

Y.C.\ was supported by the Rich Internship Program at CCNY.

G.T.\ was supported in part by the RTG award grant (DMS-2134107) from the NSF.
%%%%%%%%%%%%%%%%%%%%%%%%%%%%%
\subsection{Disclosure of generative-AI tool use.}
%%%%%%%%%%%%%%%%%%%%%%%%%%%%%
During the preparation of this work, we used generative-AI systems (OpenAI GPT 5.6 Sol and 6 Astra) to assist with presentation and polishing of the proofs.
All AI-generated suggestions were carefully fleshed out and checked by the authors, who take full responsibility for the results and the correctness of the proofs.

%%%%%%%%%%%%%%%%%%%%%%%%%%%%%
\section{Preliminaries}\label{sec:prelim}
%%%%%%%%%%%%%%%%%%%%%%%%%%%%%
In this section, we briefly recall relevant definitions and key examples. 
Let $G=(V,E)$ be an infinite connected locally finite graph and let $\Gamma$ be a closed subgroup of $\Aut(G)$ that acts quasi-transitively on $G$. For simplicity, we often restrict to the transitive case, but the results could be extended to the quasi-transitive setting by appropriate adjustments of our arguments.

A (bond) \textbf{percolation} process on a graph $G=(V,E)$ is a probability measure $\mathbf{P}$ on $2^E$. Percolation is called $\Gamma$-invariant if $\mathbf{P}$ is invariant under the action of $\Gamma$.
We refer to elements $\omega \in 2^E$ as \textbf{configurations} and the connected components of $\omega$ are called \textbf{clusters}.
For $p \in [0,1]$, a bond percolation process on $G$ is called $\mathrm{Bernoulli}(p)$ if every edge is present in a configuration independently with probability $p$. 

Recall that a locally compact group $\Gamma$ is \textbf{unimodular} if its left Haar measure is right-invariant; in the context of groups of automorphisms of a locally finite graph, there is an equivalent, more combinatorial, formulation. First, we define the induced \textbf{relative weight function} $\w_\Gamma:V^2\to\bR^+$ on $G$, for $x,y \in V$, by
\begin{equation}\label{def:haarweights}
\w_\Gamma^y(x):=\abs{\Gamma_x y}/\abs{\Gamma_y x},
\end{equation}
where $\Gamma_v := \{\gamma\in\Gamma \mid \gamma v=v\}$ is the stabilizer of $v\in V$.
Switching the reference point $y$ in the relative weight function $\w_\Gamma^y$ changes
its value only by a multiplicative factor; thus it does not affect the orders $\prec_\w$ and $\prec_\beta$.
Similarly, all of the definitions below are also independent of the reference point. When the dependency on the reference point is not important to highlight, we omit it from the notation.

By \cite{Trofimov} $\Gamma$ is unimodular if and only if $\w_\Gamma$ is constant on each $\Gamma$-orbit. Hence, when $\Gamma$ is nonunimodular, subsets of $V$ can be of two different quantitative types: of finite total weight, to which we refer as \textbf{light}, or of infinite total weight, to which we refer as \textbf{heavy}. This yields a particularly interesting phenomenon in percolation: heavy clusters often behave like infinite clusters in the unimodular setting, while infinite light clusters sometimes exhibit properties unique to the nonunimodular setting.

As we mentioned in the introduction, one of the key applications of the free minimal spanning forest is to derive an invariant infinitely ended subforest of the graph of interest. Unfortunately, this property is much less powerful in the nonunimodular setting, and in \cite{CTT22,AnushRobin} the authors introduced the following notion of the weight of an end. First we say that an infinite set of vertices $A$ is $\w_\Gamma$-\textbf{vanishing} if any sequence of distinct vertices $(x_n)_{n\in \bN}\subseteq A$ satisfies $\w_{\Gamma}^o(x_n)\to 0$ for some $o\in V$; otherwise we call $A$ $\w_\Gamma$-\textbf{nonvanishing}. 
We say that an end $\xi$ of $G$ is $\w_\Gamma$-nonvanishing if $G$ admits a $\w_\Gamma$-nonvanishing sequence of vertices that converges to $\xi$. See \cite[Sections 2.4 and 5]{TT25} for discussion on various notions of weighted ends in applications to percolation theory.
We say that a finite connected subgraph $F\subseteq G$ is a \textbf{trifurcation} (resp.\ $\w$-\textbf{trifurcation}) if removal of $F$ from $G$ creates at least $3$ new infinite components (resp.\ $\w$-nonvanishing components). When $F$ is a single vertex, we call it a \textbf{trifurcation-vertex}.

Next, one of the main tools in the field is the Mass Transport Principle (MTP), stated as follows.

\begin{thm}[(weighted) Mass Transport Principle]\label{thm:MTP}
Let $G$ be a connected locally finite graph, $\Gamma$ be a closed subgroup of $\Aut(G)$ that acts transitively on $G$, and $\w_\Gamma$ be the induced relative weight function as in \eqref{def:haarweights}. Then for any function $f:V^2 \to \bR^+$ that is invariant under the diagonal action of $\Gamma$ we have for any $o\in V$
\[
\sum_{v\in V} f(o,v)=\sum_{v\in V} f(v,o)\w^{o}(v).
\]
\end{thm}

The following are two of the simplest examples of transitive graphs that admit a nonunimodular subgroup of automorphisms.
\begin{ex}\label{ex:nonuni-tree}
For any $r,s\in \bN$ such that $1\le r<s$, consider an $ (r+s) $- regular tree and an orientation of edges such that every vertex has exactly $r$ outgoing edges and $s$ incoming edges. Let $\Gamma$ be the subgroup of the automorphisms of the tree that preserves the orientation of the edges. Under such $\Gamma$ each vertex $x\in V$ has $r$ neighbors of weight $\w_\Gamma^x(\cdot)=s/r$ and $s$ neighbors of weight $\w_\Gamma^x(\cdot)=r/s$.
\end{ex}

\begin{ex}[The grandparent graph]\label{ex:gp}
    The grandparent graph $\mathrm{GP}(k)$, $k\geq 2$, is constructed as follows. Starting from $T_{1,k}$, connect every vertex to the unique vertex at distance $2$ using outgoing edges. Notice that neighborhood of a vertex $x$ consits of a single parent ($\w_\Gamma^x(\cdot)=k$), a single grandparent ($\w_\Gamma^x(\cdot)=k^2$), $k$ children ($\w_\Gamma^x(\cdot)=1/k$), and $k^2$ grandchildren ($\w_\Gamma^x(\cdot)=1/k^2$). 
\end{ex}

Finally, the \textbf{local (a.s.)\ limits} of percolation processes are taken in the product topology on $2^E$. Thus, by local convergence of a random forest we mean that, almost surely, it eventually agrees with the limiting random forest on every finite set of edges. Notice that this notion of convergence is stronger than local weak convergence, which is only concerned with convergence of the laws on finite subsets of $E$.

%%%%%%%%%%%%%%%%%%%%%%%%%%%%%
\section{Main results}\label{sec:main_res}
%%%%%%%%%%%%%%%%%%%%%%%%%%%%%
 In our first result, we show that for $\mathfrak{F}_{\beta}(\omega)$ the statement of \cite[Theorem 1.10]{CTT22} holds without the assumption of the existence of a $\w$-trifurcation vertex.

\begin{thm}\label{thm:main_CTT}
Let $G$ be a locally finite connected graph, $\Gamma$ be a closed nonunimodular subgroup of $\Aut(G)$ that acts transitively on $G$, and $\w_\Gamma$ as in \eqref{def:haarweights}. Let $\mathbf{P}$ be a $\Gamma$-invariant percolation on $G$. Let $\mvU$ be the $\textrm{Uniform}[0,1]$ labels on edges of $G$ independent of everything else.

Then for $\mathbf{P}$-a.e.\ configuration $\omega$, for every heavy cluster $C \subseteq\omega$ with $\ge3$ $\w$-nonvanishing ends, for every $\beta>0$ the random forest $\mathfrak{F}_{\beta}(\omega,\mvU)$ has a tree $T \subseteq C$ with infinitely many (and no isolated) $\w_\Gamma$-nonvanishing ends a.s.

Moreover, if $\mathbf{P}$ is insertion tolerant and is such that a.e.\ configuration contains a cluster with $\ge3$ $\w$-nonvanishing ends, then the conclusion holds for every heavy cluster.
\end{thm}

%%%%%%%%%%%%%%%%%%%%%%%%%%%%%
\subsection{Local limits of $\beta$-skewed maximal spanning forests}
%%%%%%%%%%%%%%%%%%%%%%%%%%%%%
The next set of statements concerns local limits of the $\beta$-Skewed Free and Wired Maximal Forests as $\beta$ tends to zero or infinity. While it is natural to interpret these forests as finite-temperature models interpolating between $\fmsf/\wmsf$ and $\fmax/\wmax$, we will see that the limiting behavior is not always as predictable as one might expect. We start with the statement for the wired version.

\begin{thm}[Local limit for $\mathfrak{W}_\beta$]\label{thm:local_lim_W}
    Let $G$ be a locally finite connected graph, $\Gamma$ be a closed nonunimodular subgroup of $\Aut(G)$ that acts transitively on $G$, and $\w_\Gamma$ as in \eqref{def:haarweights}.  Then 
    \begin{enumerate}
        \item\label{thm:local_lim_W_small} as $\beta\searrow 0$ the local limit of $\mathfrak{W}_\beta(G,\mvU)$ exists and a.s.\ satisfies
        \[
        \wmsf(G,\mvU)\subsetneq  \mathfrak{W}_{0+}(G,\mvU)\subseteq\fmsf(G,\mvU).
        \]
        \item\label{thm:local_lim_W_large} as $\beta\to \infty$ the local limit of $\mathfrak{W}_\beta(G,\mvU)$ exists and is given by $\wmax(G,\mvU)$.
    \end{enumerate}
\end{thm}

In the case of the free version, the situation turns out to be more intricate. It is easy to see that $\mathfrak{F}_\beta$ for both $\beta\searrow 0$ and $\beta\to \infty$ admits subsequences that converge weakly. Moreover, the proof of the next theorem will show that these subsequential limits must be 'sandwiched' by $\mathfrak{W}_{0+}$ and $\fmsf$, and respectively, by $\wmax$ and $\fmax$. However, it is not clear whether the limits exist in the generality of Theorem~\ref{thm:local_lim_W}. In the next two theorems, we prove the local limits of $\mathfrak{F}_\beta$ under additional assumptions.

\begin{thm}[Local limit for $\mathfrak{F}_\beta$]\label{thm:local_lim_F}
    Let $G$ be a locally finite connected graph, $\Gamma$ be a closed nonunimodular subgroup of $\Aut(G)$ that acts transitively on $G$, and $\w_\Gamma$ as in \eqref{def:haarweights}. 
    \begin{enumerate}
        \item If $\mathfrak{W}_{0+}(G,\mvU)=\fmsf(G,\mvU)$ then as $\beta\searrow 0$ the local limit of $\mathfrak{F}_\beta(G,\mvU)$ exists and is given by $\fmsf(G,\mvU)$.
        \item If $\wmax(G,\mvU)=\fmax(G,\mvU)$ then as $\beta\to \infty$ the local limit of $\mathfrak{F}_\beta(G,\mvU)$ exists and is given by $\fmax(G,\mvU)$.
    \end{enumerate}
\end{thm}

\begin{thm}\label{thm:local_lim_F_large_discr}
    Let $G=(V,E)$ be a connected locally finite graph and $\w:V\to\bR^+$ be a weight function taking values in $\gl^{\bZ}$ for some $\gl>1$. Then as $\beta\to \infty$ the local limit of $\mathfrak{F}_\beta(G,\mvU)$ exists and is given by $\fmax(G,\mvU)$.
\end{thm}

\begin{ques}\label{ques:local_lim_F_large}
    In the setting of Theorem~\ref{thm:local_lim_W}, as $\beta\to \infty$, does the local limit of $\mathfrak{F}_\beta(G,\mvU)$ exist and is it always equal to $\fmax(G,\mvU)$ a.s.?
\end{ques}

Theorems~\ref{thm:local_lim_F} and~\ref{thm:local_lim_F_large_discr} might lead one to believe that establishing the limit as  $\beta\searrow0$ is merely a technical challenge, and $\mathfrak{F}_{0+}$ also should always exist and be given by $\fmsf$. 
However, our next observation shows the situation is more subtle.

\begin{obs}\label{prop:prod_small_beta}\hfill
\begin{enumerate}
    \item\label{prop:prod_small_beta_cart_prod} Let $G_1$ be the Cartesian product of the grandparent graph $\mathrm{GP}(2)$ and the $d$-regular tree $T_{d}$, with $d\ge18$. As $\beta\searrow 0$ the local limit of $\mathfrak{F}_\beta(G_1,\mvU)$ exists and is given by $\mathfrak{W}_{0+}(G_1,\mvU)$ a.s. 
    \item\label{prop:prod_small_beta_T24} Let $G_2=T_{2,4}$ as in Example~\ref{ex:nonuni-tree}, then as $\beta\searrow 0$ the local limit of $\mathfrak{F}_\beta(G_2,\mvU)$ exists and is also given by $\fmsf(G_2,\mvU)$.
\end{enumerate}
        Moreover, for both $i\in\{1,2\}$ we have $\mathfrak{W}_{0+}(G_i,\mvU)\neq\fmsf(G_i,\mvU)$ and the weight functions take values in $2^\bZ$.
\end{obs}

We briefly comment on choices of the graphs in Observation~\ref{prop:prod_small_beta}. In case of $G_1$ we chose $d\ge18$ based on our computation in \eqref{T_d_ph_pu} that is needed to conclude that $\mathfrak{W}_{0+}(G_1,\mvU)\neq\fmsf(G_1,\mvU)$. We did not attempt to improve the value, although we believe that the statement is true for any $d\ge3$. 
The choice of $T_{2,4}$ was made only to ensure that the induced weight functions take values in $2^\bZ$ (so that the example satisfies assumptions in Theorem~\ref{thm:local_lim_F_large_discr}) and so that each vertex has at least $2$ outgoing edges allowing us to invoke Theorem~\ref{thm:ph_vs_pu}.

%%%%%%%%%%%%%%%%%%%%%%%%%%%%%
\subsection{Characterization of the heavy-cluster nonuniqueness phase}
%%%%%%%%%%%%%%%%%%%%%%%%%%%%%

One of the most studied features of Bernoulli$(p)$ percolation on quasi-transitive graphs are the phase transitions at which an infinite cluster emerges and then becomes unique. In the presence of the relative weight function as in \eqref{def:haarweights}, there is another phase transition at which a heavy cluster appears. The critical thresholds for the respective phase transitions are defined as follows:
\begin{align*}
p_c(G)&:=\inf\{p \in [0,1] : \mathbb{P}_p(\text{there is an infinite cluster})=1\},\\
p_h(G,\Gamma)&:=\inf\{p\in[0,1]: \mathbb{P}_p(\text{there is a $\w_\Gamma$-heavy cluster})=1\},\\
p_u(G)&:=\inf\{p \in [0,1] : \mathbb{P}_p(\text{there is exactly one infinite cluster})=1\}. %\label{eq:pu}
\end{align*}
Recall that, by \cite[Theorem 4.1.6]{Haggstrom99}, in Bernoulli$(p)$ percolation on quasi-transitive graphs, light-infinite and heavy clusters cannot coexist, yielding the following picture (see Figure~\ref{Fig:phases}). 
\begin{figure}[htb]
\begin{center}
\begin{tikzpicture}[thick, scale=0.7]

\draw [-] (-12,0) to (11,0);

\draw (-12,0.3) to (-12,-0.3);
\draw (-12,-0.3) node[below] {$0$};
\draw (-7.5,0.3) to (-7.5,-0.3);
\draw (-7.5,-0.3) node[below] {$p_c(G)$};
\draw (-0.5,0.3) to (-0.5,-0.3);
\draw (-0.5,-0.3)node[below] {$p_h(G,\Gamma)$};
\draw (6.5,0.3) to (6.5,-0.3);
\draw (6.5,-0.3) node[below] {$p_u(G)$};
\draw (11,0.3) to (11,-0.3);
\draw (11,-0.3) node[below] {$1$};

\draw (-10,1) node[above] {There are only};
\draw (-10,0.3) node[above] {finite clusters};
\draw (-4,1) node[above] {$\exists$ $\infty$-many infinite clusters};
\draw (-4,0.3) node[above] {all of which are $\w_\Gamma$-light};
\draw (3,1) node[above] {$\exists$ $\infty$-many infinite clusters};
\draw (3,0.3) node[above] {all of which are $\w_\Gamma$-heavy};
\draw (8.8,1) node[above] {$\exists\,!$ infinite cluster};
\draw (8.8,0.3) node[above] {and it is $\w_\Gamma$-heavy};
\end{tikzpicture}
\end{center}
\caption{The four possible phases of Bernoulli$(p)$ percolation, distinguished by the number and weight of infinite clusters.}
\label{Fig:phases}
\end{figure}
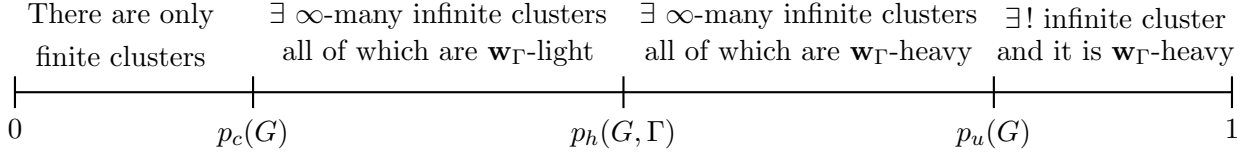

Characterizing the non-triviality of the phase with infinitely many infinite clusters $(p_c(G),p_u(G))$ has been one of the central questions in the field. Benjamini and Schramm conjectured that for a quasi-transitive graph $G$, this phase is not empty if and only if $G$ is nonamenable \cite[Conjecture 6]{BSbeyond}. Recall that we say that the graph is amenable if the Cheeger constant $\Phi_V(G)=0$, where
\begin{equation*}
    \Phi_V(G):=\inf_{K\subseteq V\atop \abs{K}<\infty}\frac{\abs{\partial_V K}}{\abs{K}}.
\end{equation*}
Notably, by the Soardi--Woess--Salvatori theorem \cite{Soardi-Woess,Salvatori}, for unimodular quasi-transitive graphs this definition coincides with the amenability of the group $\Aut(G)$.

The forward direction of \cite[Conjecture 6]{BSbeyond} was established in \cite{pc=pu}; however, the other direction is still only known in particular cases \cite{BSbeyond,PSN00,Schonmann01,Nachmias12,ThomPSN,Hutchcroft20}.

A celebrated result of  Hutchcroft \cite{Hutchcroft20} gives that whenever $\Gamma$ is nonunimodular $p_c(G)<p_h(G,\Gamma)$. On the other hand, understanding when $p_h(G,\Gamma)<p_u(G)$ remains open and is largely parallel to the question of when $p_c(G)<p_u(G)$. While nonunimodular graphs are necessarily nonamenable, the relevant geometric condition remains to be the amenability of the group $\Gamma$ or, equivalently, the weighted-amenability of the graph. We briefly recall the definition here and refer the reader to \cite{TT25} for further discussion. We say that $G$ is weighted-amenable (or $\w_\Gamma$-amenable) if 
\begin{equation*}%\label{eq:wCheeger}
    \Phi^{\w_\Gamma}_V(G):=\inf_{F\subseteq V\atop \abs{F}<\infty}\frac{\w_\Gamma^o(\partial_V F)}{\w_\Gamma^o(F)}=0.
\end{equation*}
By \cite[Theorem~3.9]{BLPS99inv}, which generalizes the Soardi--Woess--Salvatori theorem to the setting of all quasi-transitive graphs, $\w_\Gamma$-amenability of a quasi-transitive graph is equivalent to the amenability of the group $\Gamma$.
Finally, in \cite[Conjecture 6.2]{TT25}, Terlov and Tim\'ar conjectured that for a quasi-transitive graph $G$, $p_h(G,\Gamma)<p_u(G)$ if and only if $G$ is $\w_\Gamma$-nonamenable, and showed the forward direction (analogous to the result of \cite{pc=pu}).

In \cite[Proposition 3.6]{LPS06msf}, Lyons, Peres, and Schramm established the equivalence of $p_c(G)=p_u(G)$ and $\fmsf(G,\mvU)=\wmsf(G,\mvU)$. Motivated by this statement and the parallels between the applications of $\fmsf$ on unimodular graphs and $\fmax$ on nonunimodular ones, Terlov and Tim\'ar asked whether an analogous equivalence holds for $\w$-maximal spanning forests.

\begin{ques}[{\cite[Question 7.4]{TT25}}]\label{ques:ph}
    Let $\Gamma$ be a closed subgroup of $\Aut(G)$ that acts transitively on $G$.  Is it true that $p_h(G,\Gamma)=p_u(G)$ if and only if $\fmax(G,\mvU)=\wmax(G,\mvU)$?
\end{ques}
Note that since amenability of $G$ implies that $p_c(G)=p_u(G)$, it also implies that $\fmsf(G,\mvU)=\wmsf(G,\mvU)$. The next proposition shows the weighted analog of this implication.

\begin{prop}\label{prop:amen_F=W}
     Let $G$ be a locally finite connected graph, $\Gamma$ be a closed nonunimodular subgroup of $\Aut(G)$ that acts transitively on $G$, and $\w_\Gamma$ as in \eqref{def:haarweights}. If $G$ is $\w$-amenable then, for every $\beta>0$,  $\mathfrak{F}_{\beta}(G,\mvU)=\mathfrak{W}_{\beta}(G,\mvU)$ and $\fmax(G,\mvU)=\wmax(G,\mvU)$ a.s.
\end{prop}

Hence, it might be tempting to extend the equivalence in Question~\ref{ques:ph} to $\beta$-Skewed forests, namely that $p_h(G,\Gamma)=p_u(G)$ if and only if $\mathfrak{F}_\beta(G,\mvU)=\mathfrak{W}_\beta(G,\mvU)$. Our next result gives negative answers to both versions of the question.

\begin{thm}\label{thm:prod}
    Let $G$ be the Cartesian product of the grandparent graph $\mathrm{GP}(2)$ and the $d$-regular tree $T_d$, $d\ge2$, and $\Gamma=\Aut(G)$. Then, for every $\beta>0$, $\mathfrak{F}_\beta(G,\mvU)=\mathfrak{W}_\beta(G,\mvU)$ a.s.  
    
    In particular, $\fmax(G,\mvU)=\wmax(G,\mvU)$ a.s., while when $d\ge 18$, $p_h(G,\Gamma)<p_u(G)$.
\end{thm}

Although for $\beta>0$, the $\beta$-Skewed maximal forests fail to characterize the existence of a non-unique heavy cluster, the small-$\beta$ limit of the wired version still retains the relevant information about Bernoulli$(p)$ percolation, giving the following characterization.

\begin{thm}\label{thm:ph_vs_pu}
    Let $G$ be a locally finite connected graph, $\Gamma$ be a nonunimodular closed subgroup of $\Aut(G)$ that acts transitively on $G$, and $\w_\Gamma$ be as in \eqref{def:haarweights}. Then 
        \[
        \mathfrak{W}_{0+}(G,\mvU)=\fmsf(G,\mvU)\quad\mbox{a.s.}\qquad\mbox{if and only if}\qquad p_h(G,\Gamma)=p_u(G).
        \]
\end{thm}

%%%%%%%%%%%%%%%%%%%%%%%%%%%%%%%%%%%%%%%%%%%%%%%%%%%%%%%%%%
\section{Local flexibility of $\mathfrak{F}_\beta(G)$}\label{sec:flexibility_F}
%%%%%%%%%%%%%%%%%%%%%%%%%%%%%%%%%%%%%%%%%%%%%%%%%%%%%%%%%%
We call a set of edges $C\subseteq E$ a \textbf{minimal cut} if its removal from $G$ disconnects the graph and it does not admit a proper subset with the same property. The following lemma shows that $\mathfrak{F}_\beta(G)$ retains the key property of $\fmax$ in relation to the cut sets.

\begin{lem}\label{lem:cut_set}
Let $G=(V,E)$ be a connected locally finite graph, $\w:V\to\bR^+$ be a weight function. For every $\beta\in(0,\infty)$, the random forest $\mathfrak{F}_\beta(G)$ intersects every $\w$-vanishing minimal cut a.s.
\end{lem}
\begin{proof}
Restrict to the probability one event on which the random variables $\{X_{e,\beta}:e\in E\}$ are positive and pairwise distinct. Let $C\subseteq E$ be any $\w$-vanishing minimal cut. 

Fix $e\in C$. Since $C$ is $\w$-vanishing and $X_{e,\beta}\le\w(e)^\beta,$ there are only finitely many edges $e'\in C$ such that $X_{e',\beta}\ge X_{e,\beta}.$ It follows that $C$ admits a unique $\prec_\beta$-largest edge; denote it by $e_*$.  Next we claim that $e_*\in\mathfrak{F}_\beta(G)$.
Indeed, if $e_*$ were deleted, there would be a cycle on which $e_*$ is the $\prec_\beta$-smallest edge. Since $C$ is a minimal cut, every cycle
containing $e_*$ contains another edge $f\in C$, but by the choice of $e_*$, $f\prec_\beta e_*$, yielding that $e_*$ cannot be is the $\prec_\beta$-smallest edge on the cycle. Therefore $e_*\in C\cap\mathfrak{F}_\beta(G)$.
\end{proof}

The following follows immediately from Lemma~\ref{lem:cut_set}.
\begin{cor}\label{cor:F_nonvan}
    Let $G=(V,E)$ be a connected locally finite graph and $\w:V\to\bR^+$ be a weight function. For any $\beta>0$ assume that $\mathfrak{F}_\beta(G)$ is disconnected, then it contains only $\w$-nonvanishing trees a.s.
\end{cor}

We will later prove the same statement for $\mathfrak{W}_\beta(G)$ via a completely different argument (see Lemma~\ref{lem:W_nonvan}). Notably, since $\mathfrak{W}_\beta(G)\subseteq\mathfrak{F}_\beta(G)$, it actually provides an alternative proof of this corollary.

Our next result is a set of key properties (similar to those in \cite[Lemma 7.3]{BLPS99inv}) that hold for $\mathfrak{F}_\beta(G)$ with any fixed value of the parameter $\beta$. We would like to highlight that similar to $\fmsf$ the last part of this proposition is enabled by the fact that for any finite subtree $T\subseteq G$ labels $\mvU$ on $G$ could be altered in finitely many places such that $T\subseteq \mathfrak{F}_\beta(G)$, which is not true for the $\fmax$, and was one of the original motivations for the introduction of the $\beta$-Skewed version of the random forest.  

\begin{prop}\label{prop:F_prop}
Let $G=(V,E)$ be a connected locally finite graph, $\w:V\to\bR^+$ be a weight function. For $\beta\in(0,\infty)$ let $\mathfrak{F}_\beta(G)$ be the $\beta$-Skewed Free Maximal Spanning Forest on $G$. For every $v\in V$ let $T(v)$ denote the tree in $\mathfrak{F}_\beta(G)$ that contains $v$. The following properties hold:
\begin{enumerate}
    \item\label{main_prop_heavy} $\forall v\in V$ the tree $\w(T(v))=\w(V)$ a.s.
    \item \label{main_prop_size} $\forall v\in V$ the tree $\abs{T(v)}=\abs{V}$ a.s.
    \item\label{main_prop_ends} $\forall v\in V$ and $k\in \bZ^+$, if $G$ has $\ge k$ ends (resp.~$\w$-nonvanishing ends), then $T(v)$ also has $\ge k$ ends (resp.~$\w$-nonvanishing ends) with positive probability.
\end{enumerate}
\end{prop}
\begin{proof}\hfill

    \textbf{Part~\eqref{main_prop_heavy}.} For every $\beta\in(0,\infty)$ Corollary~\ref{lem:cut_set} implies that either $\mathfrak{F}_\beta(G)$ is connected, in which case obviously, or $\w(T(v))=\w(V)$ or every tree is nonvanishing. In the later case both weights have to be infinite and thus equality holds.
    
    \textbf{Part~\eqref{main_prop_size}.} Since heavy sets are necessarily infinite, the proof follows from the same argument as part~\eqref{main_prop_heavy}.
    
    \textbf{Part~\eqref{main_prop_ends}} If $G$ has $\ge k$ ends then it admits a connected finite subgraph $F$ such that $G\setminus F$ has at least $k$ infinite components. Let $m$ be the smallest weight of a vertex in $F$ and $T_F$ be a spanning tree on $F$. Consider an event $A_F$ where labels in $\mvU$ satisfy the following two properties:
    \begin{enumerate}
        \item $1-U_e>\frac12\left(\frac{m}{\w(e)}\right)^{\beta}$ for every $e\in T_F$,\\
        \item $1-U_e<\frac12\left(\frac{m}{\w(e)}\right)^{\beta}$ for every $e\notin T_F$ such that at least one endpoint of $e$ is in $T_F$.
    \end{enumerate}
    It follows that on such an event $T_F\subset \mathfrak{F}_\beta(G)$; indeed any cycle that contains an edge from $T_F$ must contain an edge $e\notin T_F$ such that at least one of its endpoint in $T_F$.
    
    Let $S$ be any of the connected components created by the removal of $F$ from $G$ and let $\partial_E(S,F)$ be the set of edges between $F$ and $S$ in $G$. 
    Notice that on the event $A_F$ the random forest $\mathfrak{F}_\beta(G)\cap (S\cup F\cup \partial_E(S,F))$ coincides with $\mathfrak{F}_\beta(S\cup F\cup \partial_E(S,F))$.
    In particular, $\mathfrak{F}_\beta(G)\cap S$ is finite (resp.\ vanishing) if and only if  $S$ is finite (resp.\ vanishing). The conclusion now follows from the assumption that removal of $F$ creates at least $k$ infinite (resp.\ nonvanishing) and local finiteness. Finally to derive the desired statement for an arbitrary vertex $v\in V$, enlarge $F$ so that $v\in F$.
\end{proof}

\begin{proof}[Proof of Theorem~\ref{thm:main_CTT}]
As we seen in the proof of Proposition~\ref{prop:F_prop} for each cluster $C$ with $\ge3$ $\w$-nonvanishing ends, for $\w$-trifurcation $F$ of $C$ there is an event $A_F$ of positive probability, on which $\mathfrak{F}_\beta(C)$ contains a tree $T$ with $\ge3$ $\w$-nonvanishing ends. A standard application of the MTP, then implies that each such tree has infinitely many and no isolated $\w$-nonvanishing ends. It remains to improve to the `almost sure' statement. 

For each cluster $C$ with $\ge3$ $\w$-nonvanishing ends let $k_C$ be the smallest natural number such that there is a $\w$-trifurcation $F\subset C$ of size at most $k_C$ such that for some spanning tree $T_F\subseteq F$ $\pr(A_F)\ge k_C^{-1}$.

Let $\cF(C)$ be an invariant maximal set of $\w$-trifurcations that satisfy these condition are are pairwise at least distance $2$ (in $C$) from each other. For each $F\in \cF$ choose an admissible tree $T_F$ uniformly at random (independently of everything else). Since $\cF(C)$ is non-empty and $C$ is heavy, by MTP, it has to be infinite. Now, conditional on $C$, by disjointness the events $(A_F)_{F\in\cF(C)}$ are independent. Thus, after arbitrarily enumerating elements in $\cF(C)$, we get
\[
\pr\left(\bigcap_{i=1}^n A_{F_i}^c \mid C\right)\le \left(1-\frac{1}{k_C}\right)^n\to0.
\]

The moreover part follows from the same argument as in the proof of \cite[Theorem 4.15]{CTT22}. There the deletion tolerance was used only to create a $\w$-trifurcation vertex. The finite-subtree forcing property of $\mathfrak{F}_\beta$, removes the need for it.
\end{proof}

%%%%%%%%%%%%%%%%%%%%%%%%%%%%%%%%%%%%%%%%%%%%%%%%%%%%%%%%%%
\section{Local limits of $\mathfrak{W}_\beta$}\label{sec:loc_lim_W}
%%%%%%%%%%%%%%%%%%%%%%%%%%%%%%%%%%%%%%%%%%%%%%%%%%%%%%%%%%
In this section we prove Theorem~\ref{thm:local_lim_W}. For the purposes of presentation we will derive the statement into two separate propositions that teat small $\beta$ and large $\beta$ limits respectively. In fact, in the former case we prove something stronger, as we explicitly characterize $\mathfrak{W}_{0+}$, see Proposition~\ref{prop:local_lim_W_small}.

We start by the following lemma that will be used repeatedly throuout this section.

\begin{lem}\label{order_alignment}
     Let $G=(V,E)$ be a connected locally finite graph, $\w:V\to\bR^+$ be a weight function. For any finite set of edges $H\subseteq E$ then there exists  $A_H(\mvU),B_H(\mvU)<\infty$, such that restricted to $H$
     \begin{enumerate}
         \item\label{order_alignment_0} the $\prec_{\beta}$ order agrees with $\prec_{0}$ order for every $\beta<A_H(\mvU)$.
         \item\label{order_alignment_infty} the $\prec_{\beta}$ order agrees with $\prec_{\w}$ order for every $\beta>B_H(\mvU)$.
     \end{enumerate}
\end{lem}
\begin{proof}
Consider any pair of edges $e,e'\in H$.    

\textbf{Case 1: $\w^o(e')=\w^o(e)$}. In this case $e\prec_{\w} e'$ whenever $e\prec_0 e'$, that is $1-U_e<1-U_{e'}$; we have that $e\prec_\beta e'$ for every $\beta\ge 0$. If all pairs of edges in $H$ fall into this case set, $A_H(\mvU)=1$ and $B_H(\mvU)=0$.

\textbf{Case 2: $\w^o(e')>\w^o(e)$.} Define
        \begin{equation}\label{def:phi_ratio}
        \varphi_{e,e'}(\beta) :=\frac{X_{e',\beta}}{X_{e,\beta}}=\left(\frac{\w^o(e')}{\w^o(e)}\right)^\beta \frac{1-U_{e'}}{1-U_e}. 
        \end{equation} 
    Note that $\varphi_{e,e'}(\beta)>1$ if and only if $e\prec_\beta e'.$ 
    
    Towards part~\eqref{order_alignment_infty}, since $\w^o(e')/\w^o(e)>1$, the function $\varphi_{e,e'}$ is increasing and tends to infinity as $\beta\to\infty$. In particular, there exists $B_{e,e'}(\mvU)<\infty$ such that, for every $\beta>B_{e,e'}(\mvU)$, we have $\varphi_{e,e'}(\beta)>1$.

    To prove part~\eqref{order_alignment_0}, first suppose that $e\prec_0e'$. Then $\varphi_{e,e'}(0)>1$, and hence, by monotonicity, $\varphi_{e,e'}(\beta)>1$ for every $\beta\ge0$. Suppose instead that $e'\prec_0e$. Then $\varphi_{e,e'}(0)<1$, and continuity of $\varphi_{e,e'}$ implies that there exists $A_{e,e'}(\mvU)>0$ such that $\varphi_{e,e'}(\beta)<1$ for every $\beta<A_{e,e'}(\mvU)$. Thus, in either case, there exists $A_{e,e'}(\mvU)>0$ such that the $\prec_\beta$ and $\prec_0$ orders of $e$ and $e'$ agree for every $\beta<A_{e,e'}(\mvU)$.

    Setting 
    \[
        A_H(\mvU)= \min\{A_{e,e'}(\mvU):e,e'\in H,\ \w^o(e')>\w^o(e)\}
    \]
    and
    \[
        B_H(\mvU)= \max\{B_{e,e'}(\mvU):e,e'\in H,\ \w^o(e')>\w^o(e)\},
    \]
    completes the proof.
\end{proof}

Recall the monotone coupling of Bernoulli$(p)$ percolations on $G$. For each $p\in[0,1)$, starting from a sequence of i.i.d.~$\mathrm{Uniform}[0,1]$ random variables $\mvU=(U_e)_{e\in E}$ we define a configuration $G[p]=\{e\in E: U_e<p\}$. Clearly, each edge is present in $G[p]$ independently with probability $p$. Note that usually in the definition of $G[p]$ the inequality is not strict, however, for out purposes it is easier to exclude the edges whose labels are exactly $p$.
Let $C_p(x)$ denote the cluster of $x$ in $G[p]$. Motivated by the proof of \cite[Proposition 3.6]{LPS06msf} next we show component characterization for an edge belonging to a wired maximal spanning forest produced by an abstract linear order $\prec$ on the edges $E$. First, given such an order $\prec$ for an edge $e\in E$ define the upward cone
    \[
    \mathrm{Cone}(e,\prec):=(V,\{e'\in E: e\prec e'\}).
    \]

\begin{lem}\label{lem:wired-characterization}
    Let $G=(V,E)$ be a connected locally finite graph and $\prec$ be a linear order on $E$. Suppose $\mathfrak{W}_{\prec}$ is the wired $\prec$-maximal spanning forest, that is a random spanning forest acquired by deleting the $\prec$-least edge in each cycle and bi-infinite simple path.
    Then $e=(x,y)\in\mathfrak{W}_{\prec}$ if and only if $x$ and $y$ lie in different connected components of $\mathrm{Cone}(e,\prec)$ and at least one of these components is finite.
\end{lem}
\begin{proof}
    If $x$ and $y$ are connected by a path in $\mathrm{Cone}(e,\prec)$, this produces a cycle in which $e$ is the smallest with respect to $\prec$. If $x$ and $y$ are in different infinite connected components of $\mathrm{Cone}(e,\prec)$, this produces a bi-infinite simple path in which $e$ is the smallest with respect to $\prec$. Finally, any finite cycle or bi-infinite simple path that deletes $e$ must fall into one of these two cases.
\end{proof}

\begin{prop}\label{prop:local_lim_W_small}
    Let $G$ be a locally finite connected graph, $\Gamma$ be a closed subgroup of $\Aut(G)$ that acts transitively on $G$, and $\w_\Gamma$ be as in \eqref{def:haarweights}. Then $(x,y)\in \mathfrak{W}_{0+}(G,\mvU)$ if and only if the components of $x$ and $y$ are distinct in $G[U_e]$ at least on of which is $\w_\Gamma$-light.
\end{prop}

\begin{proof}
    Fix an edge $e=(x,y)$. Note that at $\beta=0$, $\mathrm{Cone}(e,\prec_0)=G[U_e],$  where the edge $e$ itself is absent due to our convention with the strict inequality in the definition of $G[p]$.

    Suppose first that $C_{U_e}(x)=C_{U_e}(y)$. Choose a finite path that connects $x$ and $y$ in $G[U_e]$. By Lemma~\ref{order_alignment} for every sufficiently small $\beta>0$, all edges on this path belong to $\mathrm{Cone}(e,\prec_\beta).$ Hence $x$ and $y$ are     connected in $\mathrm{Cone}(e,\prec_\beta)$, and
    Lemma~\ref{lem:wired-characterization} gives $e\notin\mathfrak{W}_\beta(G,\mvU).$
    
    Suppose next that $C_{U_e}(x)\neq C_{U_e}(y)$ and that, say,  $C_{U_e}(x)$ is light. Consider
    \[
        D:=\left\{e'\in\partial_E C_{U_e}(x): \w^o(e')>\w^o(e)\right\}.
    \]
    Since $C_{U_e}(x)$ is light and $G$ is locally finite the set $D$ must be finite. 
    If $f\in\partial_E C_{U_e}(x)\setminus(D\cup\{e\})$, then
    $U_f>U_e$ and $\w^o(f)\le \w^o(e)$, and consequently $X_{f,\beta}<X_{e,\beta}$ for every $\beta>0$.
    Finiteness of $D$ then yields that the same inequality holds for every $f\in D$ whenever $\beta>0$ is sufficiently small.
    Letting $C_{\prec_\beta}(x)$ denote the component of $x$ in$\mathrm{Cone}(e,\prec_\beta)$, we get that no edge of $\mathrm{Cone}(e,\prec_\beta)$ leaves $C_{U_e}(x)$, and $C_{\prec_\beta}(x)\subseteq C_{U_e}(x)$.
    Moreover, since random labels $U_f$ are in $[0,1]$ if $f\in\mathrm{Cone}(e,\prec_\beta)$, then 
    \[
    \w^o(f)>\w^o(e)(1-U_e)^{1/\beta}=:c_{e,\beta}>0.
    \]
    Therefore, apart from possibly $x$ itself, every vertex in $C_{\prec_\beta}(x)$ has weight greater than $c_{e,\beta}>0$. Since this component is contained in the light cluster $C_{U_e}(x)$, it must be finite. Finally, since it does not contain
    $y$, Lemma~\ref{lem:wired-characterization} gives $e\in\mathfrak{W}_\beta(G,\mvU)$
    for every sufficiently small $\beta>0$.

    It remains to consider the case in which $C_{U_e}(x)\neq C_{U_e}(y)$ and both are heavy. Since $U_e$ has a continuous distribution, a.s.\ $U_e\neq p_h(G,\Gamma)$, and hence $U_e>p_h(G,\Gamma).$ Choose $q\in(p_h(G,\Gamma),U_e)$. By \cite{Haggstrom99} each $C_{U_e}(x)$ and $C_{U_e}(y)$ must contain an infinite $q$-cluster. Clearly, these clusters are distinct and heavy. By \cite[Theorem 5.5]{Timar06nonu}, each of these clusters contains an infinite connected subgraph, denoted by $K_x$ and $K_y$ respectively, whose vertices belong to a finite union of levels. Note that $K_x$ and $K_y$ may not contain $x$ and $y$.
    
    Since $K_x$ lies in finitely many levels, there exists $c>0$ such that ${\w^o(e')}/{\w^o(e)}\ge c>0$ for every edge $e'\in K_x$. Moreover, every edge $e'\in K_x$ satisfies $U_{e'}<q<U_e$. Thus, recalling $\varphi_{e,e'}$ from \eqref{def:phi_ratio} we have
\[
    \varphi_{e,e'}(\beta) =\frac{X_{e',\beta}}{X_{e,\beta}}=\left(\frac{\w^o(e')}{\w^o(e)}\right)^\beta\frac{1-U_{e'}}{1-U_e}>c^\beta\frac{1-q}{1-U_e}.
\]
Since
\[
    \lim_{\beta\searrow0}c^\beta\frac{1-q}{1-U_e}=\frac{1-q}{1-U_e}>1,
\]
there exists $A_x>0$ such that $e\prec_\beta e'$ for every $e'\in K_x$ whenever $\beta<A_x$.
Now if $x\notin K_x$, choose a path in $C_{U_e}(x)$ that connects $x$ to $K_x$, by Lemma~\ref{order_alignment} for sufficiently small $\beta$, this path will be contained in the cluster $C_{\prec_\beta}(x)$. Thus for such $\beta$, $K_x\subseteq C_{\prec_\beta}(x)$ and therefore $C_{\prec_\beta}(x)$ infinite. Applying the same argument to $y$ and Lemma~\ref{lem:wired-characterization} yields that $e\notin\mathfrak{W}_\beta(G,\mvU)$ for every $\beta$ small enough.
\end{proof}

\begin{cor}\label{cor:W_zero_inclusions}
    Let $G$ be a locally finite connected graph, $\Gamma$ be a closed nonunimodular subgroup of $\Aut(G)$ that acts transitively on $G$, and $\w_\Gamma$ be as in \eqref{def:haarweights}. Then $\wmsf(G)\subsetneq \mathfrak{W}_{0+}(G,\mvU)$
\end{cor}
\begin{proof}
    Fix $e=(x,y)$ and recall that by Lemma~\ref{lem:wired-characterization} $e\in\wmsf(G,\mvU)$ if and only if the $G[U_e]$-clusters of $x$ and $y$ are distinct and at least one of them is finite. Since every finite cluster is automatically light, by
    Proposition~\ref{prop:local_lim_W_small} we have $\wmsf(G,\mvU)\subseteq \mathfrak{W}_{0+}(G,\mvU).$

    To show that the inclusion is strict, first recall that by \cite{Hutchcroft20} $(p_c(G),p_h(G,\Gamma))\neq\emptyset.$
    Let $\{e_1,\ldots,e_m\}$ be representatives of the finitely many $\Gamma$-orbits of edges. For $e_i=(x_i,y_i)$ and $p\in[0,1]$, let $A_i(p)$ be the event that $x_i$ and $y_i$ lie in distinct infinite clusters of $G[p]\setminus\{e_i\}$, that is a configuration of Bernoulli$(p)$ percolation where $e_i$ is forced to be closed.

    For every $p\in (p_c(G),p_h(G,\Gamma))$, Bernoulli$(p)$ percolation has infinitely many infinite clusters, all of which are light \cite[Theorem 4.1.6]{Haggstrom99}. The usual insertion tolerance and mass-transport argument then gives there are infinitely many edges that connect two distinct light clusters and thus
    \[
        \sum_{i=1}^m\pr\bigl(A_i(p)\bigr)>0.
    \]
    Hence, for some $i$,
    \[
        \int_{p_c}^{p_h}\pr\bigl(A_i(p)\bigr)\,dp>0.
    \]

    The event $A_i(p)$ depends only on the labels $\{U_f:f\neq e_i\}$, while $U_{e_i}$ has $\mathrm{Uniform}[0,1]$ distribution
    and is independent of these labels. Therefore, by Fubini's theorem,
    \[
        \pr\bigl(U_{e_i}\in (p_c,p_h),\,A_i(U_{e_i})\bigr)=\int_{p_c}^{p_h}\pr\bigl(A_i(p)\bigr)\,dp>0.
    \]
    On this event, the two endpoint of $e_i$ in separate $U_{e_i}$-clusters, both of which are light and infinite. Therefore
    Proposition~\ref{prop:local_lim_W_small} and Lemma~\ref{lem:wired-characterization} imply that with positive probability
    \[
    e_i\in\mathfrak{W}_{0+}(G,\mvU)\setminus \wmsf(G,\mvU).
    \]
    Since the event that $\mathfrak{W}_{0+}(G,\mvU)$ and $\wmsf(G,\mvU)$ differ is $\Gamma$-invariant and these forests are factors of i.i.d.,\ in particular they are ergodic (e.g.\ see the proof of \cite[Proposition 4.8]{CTT22}), they differ almost surely.
\end{proof}

Before proceeding to the analysis of the large-$\beta$ limit for the $\beta$-Skewed Wired Maximal Spanning forest, we present the following lemma concerning the alignment of linear orders $\prec_{\beta}$ and $\prec_{\w}$ on any finite set of edges for sufficiently large $\beta$.

\begin{lem}\label{lem:ascending_ray_alignment}
     Let $G=(V,E)$ be a connected locally finite graph, $\w:V\to\bR^+$ be a weight function. Suppose $G$ contains a simple ray $(y,x,x_1,x_2,\ldots)$, such that $\w(x)=1$ and $\w(x_n)=r^n$ for some $r>1$ and every $n\ge 1$. Then there exists $B(\mvU)<\infty$ such that, for every $\beta>B(\mvU)$, $(x,y)\prec_\beta (x_n,x_{n+1})$ for every $n\ge 1$.
\end{lem}
\begin{proof}
    Write $e=(x,y)$ and $e_n=(x_n,x_{n+1})$. Since $\w(x)=1$, we have $\w(e)\le 1$, while $\w(e_n)=r^n$. Consider the events $A_n:=\{1-U_{e_n}\le r^{-n/2}\}$. Since
    \[
        \sum_{n=1}^{\infty}\pr(A_n)=\sum_{n=1}^{\infty}r^{-n/2}<\infty,
    \]
    by the Borel--Cantelli lemma we have that, a.s.\ there exists $N<\infty$ such that $1-U_{e_n}>r^{-n/2}$ for every $n\ge N$.
    Hence, for every $\beta\ge 1/2$ and $n\ge N$,
    \[
        X_{e_n,\beta}=r^{n\beta}(1-U_{e_n})>1,
    \]
    whereas $X_{e,\beta}=\w(e)^\beta(1-U_e)\le 1.$ Thus $e\prec_\beta e_n$ for every $n\ge N$ and every $\beta\ge 1/2$.

    Finally, since $\w(e)\le 1<r^n=\w(e_n)$, Lemma~\ref{order_alignment} applied to the finite set $H=\{e,e_1,\ldots,e_{N-1}\}$
    gives $B_H(\mvU)<\infty$ such that $e\prec_\beta e_n$ for every $n\le N$ whenever $\beta>B_H(\mvU)$. Taking $B(\mvU)$ to be the maximum between $1/2$ and $B_H(\mvU)$ completes the proof.
\end{proof}

We are now ready to derive the large-$\beta$ limit of the $\beta$-Skewed Wired Maximal Spanning Forest on transitive graphs.

\begin{prop}\label{prop:local_lim_W_large}
    Let $G$ be a locally finite connected graph, $\Gamma$ be a closed subgroup of $\Aut(G)$ that acts transitively on $G$, and $\w_\Gamma$ be the $\Gamma$-invariant relative weight function as in \eqref{def:haarweights}. Then as $\beta\to \infty$ the local limit of $\mathfrak{W}_\beta(G,\mvU)$ exists and is given by $\wmax(G,\mvU)$.
\end{prop}
\begin{proof}

First we show that $e=(x,y)\in\wmax(G,\mvU)$ then it is in $\mathfrak{W}_\beta(G,\mvU)$ for every sufficiently large $\beta$. By Lemma~\ref{lem:wired-characterization} if $e=(x,y)\in\wmax(G,\mvU)$ then $x$ and $y$ lie in different connected components of $\mathrm{Cone}(e,\prec_{\w})$ one of which is finite. Assume the finite component is the one containing $x$ and denote it by $C_{\prec_{\w}}(x)$. Note that every edge $e'\in \partial_EC_{\prec_{\w}}(x)\setminus\{e\}$ satisfies $e'\prec_{\w} e$. By local finiteness, there are finitely many such edges $e'$, and hence, by Lemma~\ref{order_alignment} all of them also satisfy $e'\prec_{\beta} e$ for $\beta$ large enough. Therefore for every such $\beta$, the component of $\mathrm{Cone}(e,\prec_{\beta})$ that contains $x$ is finite and does not contain $y$. Lemma~\ref{lem:wired-characterization} now concludes the desired inclusion.

Towards the other inclusion, fix an edge $e=(x,y)\notin\wmax(G,\mvU)$. By Lemma~\ref{lem:wired-characterization}, either $x$ and $y$ are connected in $\mathrm{Cone}(e,\prec_{\w})$, or their respective components in $\mathrm{Cone}(e,\prec_{\w})$ are both infinite.

Suppose first that $x$ and $y$ are connected in $\mathrm{Cone}(e,\prec_{\w})$. Choose a finite path connecting $x$ and $y$ inside $\mathrm{Cone}(e,\prec_{\w})$. By Lemma~\ref{order_alignment}, for all sufficiently large $\beta$, every edge $e'$ in this path is larger than $e$ with respect to $\prec_\beta$. Hence $x$ and $y$ are connected in $\mathrm{Cone}(e,\prec_\beta)$, and therefore $e\notin\mathfrak{W}_\beta(G,\mvU).$

Suppose now that the components $C_{\prec_{\w}}(x)$ and $C_{\prec_{\w}}(y)$ of $x$ and $y$ in $\mathrm{Cone}(e,\prec_{\w})$ are both infinite. We will show that, the respective components $C_{\prec_{\beta}}(\cdot)$ of $x$ and $y$ in $\mathrm{Cone}(e,\prec_\beta)$ are also infinite for every sufficiently large $\beta$.

Note that every vertex of $C_{\prec_{\w}}(x)$ has weight at least $\w^x(e)$. If every vertex in this component has weight exactly $\w^x(e)$, then for every edge $e'\in C_{\prec_{\w}}(x)$, the fact that $e\prec_\w e' $ was determined only by the tie-breaker rule and hence $e\prec_\beta e'$ for every $\beta\ge0$. In particular, $C_{\prec_{\beta}}(x)$ is infinite.
Otherwise, choose a finite simple path $C_{\prec_{\w}}(x)$ that connects $x$ to the closest vertex of a larger weight; that is  $(x,x_1,\ldots,x_m,v_1)\subset C_{\prec_{\w}}(x)$ such that $\w^x(x_i)=1$ for all $i\le m$ and $\w^x(v_1)=r>1$. By transitivity, every vertex has a neighbor in $G$ whose relative weight is equal to $r$. Iterating from $v_1$, we deterministically obtain a simple ray $(x_{m},v_1,v_2,v_3,\ldots)$ such that $\w^x(v_n)=r^{n}$. The ray is simple since its weights strictly increase, and it does not contain any of the vertices from $\{x,x_1,\ldots,x_{m}\}$ by the choice of $v$.

Lemma~\ref{lem:ascending_ray_alignment} gives that $(x_{m},v_1)\prec_\beta (v_n,v_{n+1})$ for every $n\ge 1$ and every sufficiently large $\beta$. On the other hand, Lemma~\ref{order_alignment}, applied to the finite set consisting of $e$ and the edges along the path $(x,x_1,\ldots,x_m,v_1)$ yields $e\prec_\beta e'$ for every edge $e'$ that belongs to this path, again, for every $\beta$ that is sufficiently large. In particular, for such $\beta$ the component of $x$ in $\mathrm{Cone}(e,\prec_\beta)$ contains an infinite ray $(x,x_1,\ldots,x_m,v_1,v_2,\ldots)$ and hence infinite. Note that this ray is also in $C_{\prec_{\w}}(x)$ and hence is disjoint from $C_{\prec_{\w}}(y)$.

Applying this argument to both $y$, we find that, for every sufficiently large $\beta$, either $x$ and $y$ are connected in $\mathrm{Cone}(e,\prec_\beta)$ or their two components are both infinite. Lemma~\ref{lem:wired-characterization} then implies that     $e\notin\mathfrak{W}_\beta(G,\mvU).$
\end{proof}

\begin{proof}[Proof of Theorem~\ref{thm:local_lim_W}]
The statement follows from combination of Propositions~\ref{prop:local_lim_W_small} and~\ref{prop:local_lim_W_large}.
\end{proof}

%%%%%%%%%%%%%%%%%%%%%%%%%%%%%%%%%%%%%%%%%%%%%%%%%%%%%%%%%%
\section{Local limits of $\mathfrak{F}_\beta$}\label{sec:loc_lim_F}
%%%%%%%%%%%%%%%%%%%%%%%%%%%%%%%%%%%%%%%%%%%%%%%%%%%%%%%%%%

We first note an analog of Lemma~\ref{lem:wired-characterization} for the free maximal spanning forests.

\begin{lem}\label{lem:free-characterization}
 Let $G=(V,E)$ be a connected locally finite graph and $\prec$ be a linear order on $E$. Suppose $\mathfrak{F}_{\prec}$ is the free $\prec$-maximal spanning forest, that is a random spanning forest acquired by deleting the $\prec$-least edge in each cycle.
 Then $e=(x,y)\in\mathfrak{F}_{\prec}$ if and only if $x$ and $y$ lie in different connected components of $\mathrm{Cone}(e,\prec)$.
\end{lem}
\begin{proof}
    The statement follows from the observation that $e=(x,y)$ is the smallest with respect to $\prec$ on some cycle if and only if $x$ and $y$ are connected by a path in $\mathrm{Cone}(e,\prec)$.
\end{proof}

\begin{lem}\label{F_beta_exclusion}
    Let $G=(V,E)$ be a connected locally finite graph and
    $\w:V\to\bR^+$ be a weight function. Then the following hold.
    \begin{enumerate}
        \item If $e\notin \fmsf(G,\mvU)$, then there exists
        $A_e(\mvU)>0$ such that $e\notin\mathfrak{F}_\beta(G,\mvU)$ for every $0<\beta<A_e(\mvU)$.
        \item If $e\notin\fmax(G,\mvU)$, then there exists
        $B_e(\mvU)<\infty$ such that $ e\notin\mathfrak{F}_\beta(G,\mvU)$ for every $\beta>B_e(\mvU)$.
    \end{enumerate}
\end{lem}

\begin{proof}
    Both claims follow from the same argument.
    Fix $e=(x,y)$. Suppose first that $e\notin \fmsf(G,\mvU)$ (resp.\ $e\notin\fmax(G,\mvU)$). By Lemma~\ref{lem:free-characterization}, there exists a finite path connecting $x$ and $y$ inside $\mathrm{Cone}(e,\prec_0)$ (resp.\ $\mathrm{Cone}(e,\prec_{\w})$). By Lemma~\ref{order_alignment}, the orders $\prec_\beta$ and
    $\prec_0$ (resp.~$\prec_{\w}$) agree on this path for every sufficiently small $\beta>0$ (resp.\ sufficiently large $\beta$), by Lemma~\ref{lem:free-characterization} $e\notin\mathfrak{F}_\beta(G,\mvU).$ 
\end{proof}

\begin{proof}[Proof of Theorem~\ref{thm:local_lim_F_large_discr}]
    In light of Lemma~\ref{F_beta_exclusion} it remains to show that if $e\in\fmax(G,\mvU),$ then for every sufficiently large $\beta$ we have $e\in\mathfrak{F}_\beta(G,\mvU)$.
    
    Choose $\beta$ such that $\gl^{-\beta}<1-U_e$. We claim that, for every such $\beta$,
    \[
        \mathrm{Cone}(e,\prec_\beta)\subseteq\mathrm{Cone}(e,\prec_{\w}).
    \]
    Indeed, suppose that $e'\prec_{\w}e$. If $\w^o(e')=\w^o(e)$, then $1-U_{e'}<1-U_e$, and hence $e'\prec_\beta e$ for every $\beta\ge0$.

    If $\w^o(e')<\w^o(e)$, then, since $\w$ takes values in $\gl^{\bZ}$,
    \[
        \frac{\w^o(e')}{\w^o(e)}\le\frac{1}{\gl}.
    \]
    Therefore, by our choice of $\beta$ we have
    \[
        X_{e',\beta}=\w^o(e')^\beta(1-U_{e'})\le\w^o(e)^\beta\gl^{-\beta}<      \w^o(e)^\beta(1-U_e)=X_{e,\beta}.
    \]
    Thus $e'\prec_\beta e$, proving the claimed cone inclusion.

    Since $e\in\fmax(G,\mvU)$, Lemma~\ref{lem:free-characterization}
    implies that the endpoints of $e$ lie in different components of
    $\mathrm{Cone}(e,\prec_{\w})$. They therefore also lie in different
    components of the subgraph
    $\mathrm{Cone}(e,\prec_\beta)$, and hence $e\in\mathfrak{F}_\beta(G,\mvU).$
\end{proof}

\begin{proof}[Proof of Theorem~\ref{thm:local_lim_F}] 
    The statement follows immediately from the following observations:
    \begin{enumerate}
        \item  if $ e\in\mathfrak{W}_{0+}(G,\mvU),$ then for every sufficiently small $\beta>0$ $e\in\mathfrak{F}_\beta(G,\mvU)$,
        \item  if $e\notin \fmsf(G,\mvU),$ then for every sufficiently small $\beta>0$ $e\notin\mathfrak{F}_\beta(G,\mvU)$,
        \item  if $e\in\wmax(G,\mvU),$ then for every sufficiently large $\beta$ $e\in\mathfrak{F}_\beta(G,\mvU)$,
        \item  if $e\notin\fmax(G,\mvU),$ then for every sufficiently large $\beta$ $e\notin\mathfrak{F}_\beta(G,\mvU)$.
    \end{enumerate}
    To see these implications note that for every $\beta>0$ $\mathfrak{W}_\beta(G,\mvU)\subseteq\mathfrak{F}_\beta(G,\mvU)$. Thus items~(1) and~(3) follow from the local convergence of $\mathfrak{W}_\beta(G,\mvU)$ established in Theorem~\ref{thm:local_lim_W}, respectively. Items~(2) and~(4) follow from Lemma~\ref{F_beta_exclusion}.
\end{proof}

%%%%%%%%%%%%%%%%%%%%%%%%%%%%%%%%%%%%%%%%%%%%%%%%%%%%%%%%%%
\section{Characterization of the heavy-cluster nonuniqueness phase}\label{sec:char_ph_pu}
%%%%%%%%%%%%%%%%%%%%%%%%%%%%%%%%%%%%%%%%%%%%%%%%%%%%%%%%%%
Before proving Proposition~\ref{prop:amen_F=W} and Theorem~\ref{thm:ph_vs_pu} we establish a statement analogous to Proposition~\ref{prop:F_prop}.\eqref{main_prop_heavy} for $\mathfrak{W}_\beta(G)$.

\begin{lem}\label{lem:W_nonvan}
Let $G=(V,E)$ be a connected locally finite graph and $\w:V\to\bR^+$ be a weight function. For any $\beta>0$ assume that $\mathfrak{W}_\beta(G)$ is disconnected, then it contains only $\w$-nonvanishing trees a.s.
Similarly, assume $\wmax(G)$ is disconnected, then it also contains only $\w$-nonvanishing trees a.s.
\end{lem}
\begin{proof}
Restrict to the probability one event on which the random variables $\{X_{e,\beta}:e\in E\}$ are positive and pairwise distinct. 
Suppose, to the contrary, that $T$ is a proper $\w$-vanishing component of $\mathfrak{W}_\beta(G)$. Then by local finiteness, $\partial_E T$ is also $\w$-vanishing. In particular $\partial_E T$ admits a unique $\prec_\beta$-largest edge; denote this edge by $e=(x,y)$, such that $x\in T$ and $y\notin T$. It follows that, no edge of $\mathrm{Cone}(e,\prec_\beta)$ crosses the boundary of $T$, and hence the component of $x$ in $\mathrm{Cone}(e,\prec_\beta)$ is contained in $T$.

For every edge $e'\in\mathrm{Cone}(e,\prec_\beta)$ satisfies 
\[
\w(e')\ge\left(X_{e',\beta}\right)^{\frac{1}{\beta}}>\left(X_{e,\beta}\right)^{\frac{1}{\beta}}>0.
\]
Since this component is contained in the $\w$-vanishing set $T$, it must be finite. The endpoints of $e$ therefore lie in distinct components of $\mathrm{Cone}(e,\prec_\beta)$, one of which is finite. Lemma~\ref{lem:wired-characterization} then yields that $e\in\mathfrak{W}_\beta(G),$ contradicting $e\in\partial_E T$. 

The statement for $\wmax(G)$ follows from the same argument, where $e'\in\mathrm{Cone}(e,\prec_\w)$ immediately implies $\w(e')\ge\w(e)$.
\end{proof}

\begin{proof}[Proof of Proposition~\ref{prop:amen_F=W}]
    The same argument works for both types of forests, so let $\mathfrak{F}$ denote $\fmax(G)$ (resp.\ $\mathfrak{F}_\beta(G)$) and $\mathfrak{W}$ denote $\wmax(G)$ (resp.\ $\mathfrak{W}_\beta(G)$).

    Assume $\mathfrak{F}\neq\mathfrak{W}$, by Lemma~\ref{lem:W_nonvan} the forest $\mathfrak{W}$ is disconnected and contains only $\w_\Gamma$-nonvanishing components a.s. By local finiteness each tree of $\mathfrak{W}$ contains a $\w_\Gamma$-nonvanishing end. By the MTP, every tree of $\mathcal{F}$ that contains an edge in $\mathfrak{F}\setminus\mathfrak{W}$ must contain infinitely such edges and hence have infinitely many $\w_\Gamma$-nonvanishing ends. By \cite[Theorem 5.3]{TT25} $\mathfrak{F}$ is not hyperfinite (in the sense of \cite[Definition 4.1]{TT25}), and thus $G$ is also not hyperfinite. The conclusion now follows from \cite[Theorem 1.4]{TT25}.
\end{proof}

\begin{proof}[Proof of Theorem~\ref{thm:ph_vs_pu}]
    By Lemmas~\ref{lem:wired-characterization} and~\ref{lem:free-characterization}, together with Proposition~\ref{prop:local_lim_W_small}, for every edge $e=(x,y)$ we have $e\in\fmsf(G,\mvU)\setminus\mathfrak{W}_{0+}(G,\mvU)$ if and only if the $G[U_e]$-clusters of $x$ and $y$ are distinct and both heavy.
    Repeating the orbit-representative and Fubini argument from the proof of Corollary~\ref{cor:W_zero_inclusions}, such an edge exists with positive probability if and only if the set of $p\in[0,1]$ for which Bernoulli$(p)$ percolation has infinitely many heavy clusters has positive measure. Hence
    \[
        \pr\bigl(\fmsf(G,\mvU)\neq\mathfrak{W}_{0+}(G,\mvU)\bigr)>0\qquad\mbox{if and only if}\qquad p_h(G,\Gamma)<p_u(G).
    \]
    Finally, the event that the two forests differ is $\Gamma$-invariant, and hence has probability either zero or one by ergodicity.
\end{proof}

\begin{proof}[Proof of Theorem~\ref{thm:prod}]
    In the case when $d=2$, $G$ is $\w$-amenable (for example by \cite[Theorem 1.4]{TT25}) and thus the conclusion follows from Proposition~\ref{prop:amen_F=W}.

    Assume $d\ge3$ and fix $\beta>0$. Since $\mathfrak{W}_\beta(G,\mvU)\subseteq\mathfrak{F}_\beta(G,\mvU)$, it suffices to show that, for every edge $e=(x,y)$,
    \begin{equation}\label{eq:e_in_F_not_W}
        \pr\bigl(e\in\mathfrak{F}_\beta(G,\mvU)\setminus\mathfrak{W}_\beta(G,\mvU)\bigr)=0.
    \end{equation}
    Without loss of generality assume that $\w^{x}(y)\ge 1$ and thus $\w^x(x)=\w^x(e)=1$.

    Starting from an arbitrary vertex $v$, iteratively taking the unique neighbor whose relative weight is $2$ gives a ray    $(v=v_0,v_1,v_2,\ldots)$. Write $f_n=(v_n,v_{n+1})$, since $\w^v(f_n)=2^n$, we have
    \[
        \sum_{n=0}^{\infty}\pr(X_{f_n,\beta}\le X_{e,\beta})\le\sum_{n=0}^{\infty}\min\left\{1,\w^x(v)^{-\beta}2^{-n\beta}\right\}<\infty.
    \]
    By the Borel--Cantelli lemma, a tail of this ray belongs to
    $\mathrm{Cone}(e,\prec_\beta)$.
    Moreover, the tails of all such rays belong to the same component of $\mathrm{Cone}(e,\prec_\beta)$. Indeed, the projections of their first-coordinates are upward paths in $\mathrm{GP}(2)$ and thus eventually coincide. Starting from that point their second coordinates are connected by disjoint by copies of a fixed finite path in $T_d$. The weights of edges along these translates grow by a factor of $2$, so the same Borel--Cantelli argument shows that these connecting paths belong to $\mathrm{Cone}(e,\prec_\beta)$ at all sufficiently high weight-levels. By countability, this holds simultaneously for all pairs of starting vertices. Denote the common component containing these tails by $C_{\mathrm{tail}}(\beta)$.

    Next, we claim that every infinite component of $\mathrm{Cone}(e,\prec_\beta)$ is equal to $C_{\mathrm{tail}}(\beta)$.  Condition on $U_e$, and let 
    \[
    a:=\min\{a\in2^{\mathbb Z}:a^\beta>1-U_e\}.
    \]
    
    Every edge of $\mathrm{Cone}(e,\prec_\beta)$ has both endpoints with weight at least $a$ with respect to $x$. Furthermore, conditional on $U_e$, for any vertex of weight at least $a$ whose upward ray does not contain $e$, the probability that its entire upward ray belongs to the cone is at least
    \[
        c:=\prod_{n=0}^{\infty}\left(1-\frac{1-U_e}{a^\beta2^{n\beta}}\right)>0,
    \]
    here $c$ is strictly positive because $\sum_{n\ge0}2^{-n\beta}<\infty$.

    Consider $\mathrm{GP}(2)$ and take $o\in V(\mathrm{GP}(2))$ to be the first coordinate of $x$, with the weight normalization induced by $\w^x$. Call $o_n$ an $n$-th parent of $o$ if $\w^o(o_n)=2^n$ and there is a path connecting $o$ to $o_n$ along which the weights of vertices only increase. For each $n$, let $H_n$ consist of all vertices with weight at least $a$ for which $o_n$ is a $k$-th parent for some $k\ge0$. Note that the sets $H_n$ are finite, increasing, and exhaust the vertices of $\mathrm{GP}(2)$ whose weight is at least $a$.

    Let $b\in V(T_d)$ to be the first coordinate of $x$ and set $S_n$ to be the Cartesian product of $H_n$ and $\mathrm{Ball}_{T_d}(b,n).$ Then $(S_n)$ is an increasing sequence of finite sets exhausting $\{v\in G:\w^x(v)\ge a\}$. By construction, if the unique neighbor $v'$ of $v$ that has $\w^v(v')=2$ belongs to $S_n$ and $\w^x(v)\ge a$, then $v$ also belongs to $S_n$. Consequently, every upward ray starting in $\{v\in V:\w^x(v)\ge a\}\setminus S_n$ stays outside $S_n$.

    We sequentially reveal all labels on edges incident to $S_n$. Take $n$ large enough  $x,y\in S_n$. If there is no edge
    $(u,w)\in\partial_E S_n\cap\mathrm{Cone}(e,\prec_\beta)$ such that $u\notin S_n$ and $w$ is connected to $x$ by a path in
    $\mathrm{Cone}(e,\prec_\beta)$ lying inside $S_n$, then $C_{\prec_\beta}(x)$ must be finite.
    Otherwise, choose such an edge using the revealed labels. The upward ray from $u$ lies entirely outside $S_n$, so its labels
    remain independent and unrevealed. With conditional probability at least $c$, this entire ray belongs to $\mathrm{Cone}(e,\prec_\beta)$. On this event, $x$ is connected to $C_{\mathrm{tail}}(\beta)$. Thus
    \[
        \pr\biggl(|C_{\prec_\beta}(x)|=\infty,\,C_{\prec_\beta}(x)\neq C_{\mathrm{tail}}(\beta)\,\bigg|\,U_e,\,(U_f)_{f\cap S_n\neq\emptyset}\biggr)\le1-c.
    \]
    As $n\to\infty$, the revealed labels determine the entire $\mathrm{Cone}(e,\prec_\beta)$. By the martingale convergence theorem, the conditional probabilities above converge a.s.\ to the indicator of the conditioned event. Since $c>0$, that event has probability zero. The same argument applies to $y$.

    More generally, the same argument applies to every vertex $v$ of weight at least $a$. Hence, by countability, every infinite component equals $C_{\mathrm{tail}}(\beta)$ a.s. Finally, Lemmas~\ref{lem:wired-characterization} and~\ref{lem:free-characterization} then imply \eqref{eq:e_in_F_not_W}.

    The conclusion that $\fmax(G,\mvU)=\wmax(G,\mvU)$ follows immediately from the large-$\beta$ limits of $\mathfrak{W}_\beta$ or $\mathfrak{F}_\beta$ from Theorem~\ref{thm:local_lim_W}.\eqref{thm:local_lim_W_large} or Theorem~\ref{thm:local_lim_F_large_discr}, respectively. However, it is also easy to adopt the proof presented here to $\prec_{\w}$ for a direct argument.

    Finally, it remains to show that, when $d\ge18$, $p_h\left(G,\Gamma\right)<p_u\left(G\right)$. note that $G$ is $(d+8)$ regular and thus, letting $\rho(H)$ denote the spectral radius of the simple symmetric random walk on a graph $H$, by \cite[Proposition 7.34 and 7.35]{LyonsBook}
    \[
        \rho\left(G\right)\le\frac{d\rho(T_d)+8\rho(\mathrm{GP}(2))}{d+8}\le\frac{d\rho(T_d)+8}{d+8}=\frac{2\sqrt{d-1}+8}{d+8}.
    \]
    and thus by \cite[Theorem 7.32 and Lemma 7.33]{LyonsBook}
    \begin{equation}\label{T_d_pu}
        p_u(G)\ge \frac{1}{(d+8)\rho\left(G\right)}\ge \frac{1}{2\sqrt{d-1}+8}.
    \end{equation}
    At the same time we have
    \begin{equation}\label{T_d_ph}
        p_h(G,\Gamma)\le p_c(T_{d})=\frac{1}{d-1}.
    \end{equation}
    Putting \eqref{T_d_pu} and \eqref{T_d_ph} together we get that when $d\ge 18$
    \begin{equation}\label{T_d_ph_pu}
        p_h\left(G,\Gamma\right)\le\frac{1}{d-1}<\frac{1}{2\sqrt{d-1}+8}\le p_u\left(G\right).
    \end{equation}
\end{proof}

\begin{proof}[Proof of Observation~\ref{prop:prod_small_beta}] 
Part~\eqref{prop:prod_small_beta_cart_prod} follows directly from Theorems~\ref{thm:prod},~\ref{thm:ph_vs_pu}, and ~\ref{thm:local_lim_W}.
Part~\eqref{prop:prod_small_beta_T24} is trivial since $G_2=T_{2,4}$ is acyclic, and thus for any $\beta>0$
$
G_2=\mathfrak{F}_\beta(G_2,\mvU)=\fmsf(G,\mvU).
$
Finally, It is easy to see that $p_h(G_2)\le 1/2$ and $p_u(G_2)=1$ and thus by Theorem~\ref{thm:ph_vs_pu} we have $\mathfrak{W}_{0+}(G_2,\mvU)\neq\fmsf(G_2,\mvU)$.
    
\end{proof}

\bibliographystyle{alphaurl}
\bibliography{msf} 
\end{document}